\documentclass{amsart}
\usepackage{amssymb}
\usepackage{amsmath}
\usepackage{amsthm}
\usepackage{amsbsy}
\usepackage{bm}
\usepackage[dvips]{graphicx}

\theoremstyle{plain}
\newtheorem{theorem}{Theorem}[section]
\newtheorem{lemma}[theorem]{Lemma}
\newtheorem{corollary}[theorem]{Corollary}
\newtheorem{proposition}[theorem]{Proposition}

\theoremstyle{definition}
\newtheorem{definition}[theorem]{Definition}

\theoremstyle{remark}

\newcommand{\M}{\widetilde{M}}

\newcommand{\del}{\partial}

\begin{document}
\title{ The Complex of Non-separating Embedded Spheres}

\author{Suhas Pandit}
\address{Mathematics Section,\\
          The Abdus Salam International Center for Theoretical Physics,\\
       Trieste, Italy}
\email{jsuhas@gmail.com,spandit@ictp.it}

\date{\today}
 \maketitle

 \begin{abstract}
 For $n\geq 3$, we shall show that the group $Aut(NS(M))$ of simplicial automorphisms of the complex $NS(M)$ of non-separating embedded spheres in the manifold $M=\sharp_{n} S^{2}\times S^{1}$ is isomorphic to the group $Out(\mathbb F_n)$ of outer automorphisms of the free group $\mathbb F_n$, where $\mathbb F_n$ is identified with the fundamental group of the manifold $M$ up to conjugacy of the base point in $M$.
 \end{abstract}
\section{Introduction}
Let $\mathbb F_n$ be a free group of rank $n$. Let $Aut(\mathbb F_n)$ and $Out(\mathbb F_n)$ be the groups of automorphisms and outer automorphisms of $\mathbb F_n$. The group $Out(\mathbb F_n)$ is the quotient of $Aut(\mathbb F_n)$ by the subgroup of inner automorphisms of $\mathbb F_n$. These groups are important objects in geometric group theory and combinatorial group theory. These groups have been studied extensively in analogy with the fundamental groups of surfaces and the mapping class groups of surfaces, for instance, see~\cite{BFH1}, \cite{BF}, \cite{BFH2}, \cite{BH}, where these groups have been studied via a one-dimensional model which arises from regarding $\mathbb F_n$ as the fundamental group of a graph. Recall that the mapping class group of a surface is the group of isotopy classes of homeomorphisms of the surface. In order to study these groups, people are using topological and geometric objects on which these groups act.

Culler and Vogtmann ~\cite{CV}, introduced a space $CV_n$ called the \emph{outer space} of a free group of rank $n$ on which the group $Out(\mathbb F_n)$ acts with finite point stabilizers. The outer space with the action of $Out(\mathbb F_n)$ can be thought of as a free group analogous to the Teichm\"uller space of a surface with the action of the mapping class group of the surface. Culler and Morgan have constructed a compactification of outer space much like Thurston's
compactification of Teichm\"uller space of a surface, see ~\cite{CM}.

In the case of surfaces, to study mapping class groups of the surfaces and their Teichm\"uller spaces, we have objects like the curve complex associated to a surface, the complex of non-separating curves, the complex of pants decompositions. Recall that for a closed hyperbolic surface $F$, the curve complex $\mathbb C(F)$ is the simplicial complex with vertices free homotopy classes of essential simple closed curves on $F$. A finite collection of  isotopy classes of simple closed curves in $F$ is deemed to span a simplex if they can be realized disjointly (up to isotopy) in $F$.  There is a natural simpicial action of the mapping class group $\mathcal{M}ap(F)$ on the curve complex. This yields a homomorphism
$$\mathcal{M}ap(F) \to Aut(\mathbb C(F)),$$
 where $Aut(\mathbb C(F))$ is the group of simplicial automorphisms of the curve complex $\mathbb C(F)$. Ivanov proved that the group $Aut(\mathbb C(F))$
 is isomorphic to the mapping class group $\mathcal{M}ap(F)$ for connected orientable surfaces of genus at least $2$ in ~\cite{Iv2}. Ivanov's results were extended to lower genus cases by Korkmaz in ~\cite{Ko}. Feng Luo gave a new proof that the automorphisms of the curve complex are induced by homeomorphisms of the surface if the dimension of the curve complex is at least one and the surface is not a torus with two holes in ~\cite{FL1}. After Ivanov's work, mapping class group was viewed as the automorphism group of various geometric objects on surfaces, for instance, the complex of pants decompositions in \cite{Mar}, the complex of non-separating curves in \cite{Ir}, the complex of separating curves in \cite{BM} and the arc and curve complex in \cite{KP}.

In the free group context, people have considered complexes like the splitting complex, the complex of free factors of $\mathbb F_n$ (\cite{HV2}), the complex of separating embedded spheres (\cite{SlSa}). In \cite{HM}, Handle and Mosher have proved that the splitting complex of a free group of finite rank with its geodesic simplicial metric is Gromov hyperbolic. In \cite{BF1}, Bestvina and Feighn have proved that the complex of free factors of a free group of finite rank is hyperbolic. The $3$-manifold $M=\sharp_{n} S^{2}\times S^{1}$ has the fundamental group a free group of rank $n$. The manifold $M$ is used as a model for studying the free group and its automorphisms, in particular, the group of outer automorphisms of free groups. It was first used by Whitehead which involves looking at embedded spheres in $M$ (\cite{St3}). An embedded sphere in $M$ corresponds to a free splitting of the fundamental group $G=\pi_1(M)$.  The \emph{sphere complex} $\mathbb S(M)$ (see definition \ref{spc}) associated to $M$ is equivalent to the splitting complex associated to $\mathbb F_n$, where we identify $\mathbb F_n$ with  $\pi_1(M)$. There is a natural action of the mapping class group $\mathcal{M}ap(M)$ of $M$ on the sphere complex which yields an action of $Out(\mathbb F_n)$ on the sphere complex $\mathbb S(M)$. This gives a homomorphism $Out(\mathbb F_n) \to Aut(\mathbb S(M))$, where $Aut(\mathbb S(M))$ is the group of simplicial automorphisms of the sphere complex $\mathbb S(M)$. In~\cite{AS}, it is proved that the group $Aut(\mathbb S(M))$ is isomorphic to the group $Out(\mathbb F_n)$ for $n\geq 3$.

In this paper, we consider the \emph{complex $NS(M)$ of non-separating embedded spheres} in $M$ (see definition \ref{nsc}) on which the group $Out(\mathbb F_n)$ acts. An embedded sphere $S$ in $M$ is said to be \emph{non-separating} if the complement $M-S$ of $S$ in $M$ is connected.  We shall  prove that the group $Aut(NS(M))$ of simplicial automorphisms of the complex $NS(M)$ is isomorphic to $Out(\mathbb F_n)$, for $n\geq 3$ (see Theorem \ref{MainT}). It is due to Hatcher that the sphere complex contains an embedded copy of the \emph{spine} $K_n$ of the \emph{reduced outer space} $cv_n$. The complex $NS(M)$ also contains an embedded copy of the \emph{spine} $K_n$ of the reduced outer space $cv_n$.  As in ~\cite{AS}, we shall show that the spine $K_n$ is invariant under the group $Aut(NS(M))$ (Lemma \ref{spinelemma}) by showing that given any simplicial automorphism $\phi$ of $NS(M)$, it maps simplices in $NS(M)$ corresponding to  reduced sphere systems in $M$ to simplices in $NS(M)$ corresponding to reduced sphere systems in $M$ (see Lemma \ref{mainlemma}). The idea is to look at the links of such simplices in $NS(M)$. We shall observe that the link of a simplex in $NS(M)$ corresponding to a reduced sphere system in $M$ is spanned by finitely many vertices of $NS(M)$ and the link of an $(n-1)$-simplex which corresponds to a sphere system which is not reduced in $M$  is spanned by infinitely many vertices of $NS(M)$.  We call a simplex in $NS(M)$ which corresponds to a simple sphere system in $M$ as a \emph{simple simplex}. Using Lemma \ref{mainlemma}, we see that every simplicial automorphism $\phi$ of $NS(M)$ maps a simple simplex in $NS(M)$ to a simple simplex in $NS(M)$ (Lemma \ref{simplelemma}). From this, Lemma \ref{spinelemma} follows and then the rest of the proof follows as in ~\cite{AS}.

 Now, we outline the paper: In Section \ref{model}, we discuss the model $3$-manifold $M$ and some definitions. In Section~\ref{splittingsembed}, we show that embedded spheres in $M$ correspond to splittings of $G=\pi_1(M)$. This result is well known in the field (see ~\cite{He}), but here we give another proof of this in $M$ for the sake of completeness. In Section~\ref{GOG}, we show that given a graph of groups decomposition $\mathcal G$ of the fundamental group $G$ of $M$, there exist a system of $2$-spheres in $M$ which gives the same graph of groups decomposition $\mathcal G$ of $G$. In Section \ref{CNS}, we discuss complex $NS(M)$ of non-separating spheres and some facts about it. In Section \ref{Main}, we give a proof of the Theorem \ref{MainT}.

\tableofcontents

\section{Preliminaries}\label{model}
\subsection{The model $3$-manifold}\label{modelm}
Consider the 3-manifold $M = \sharp_{n} S^{2}\times S^{1} $, i.e.,  the connected sum of $n$ copies of $S^{2}\times S^{1} $. A description of $M$ can be given as follows: Consider the $3$-sphere $S^{3}$ and let $A_{i},B_{i},1\leq i\leq n$, be a collection of $2n$ disjointly embedded $3$-balls in $S^{3}$. Let $S_{i}$ (respectively, $T_{i}$) denote the boundary of $A_{i}$ (respectively, $B_{i}$). Let $P$ be the complement of the union of the interiors of all the balls $A_i$ and $B_i$. Then, $M$ is obtained from $P$ by gluing together $S_{i} $ and $T_{i}$ with an orientation reversing diffeomorphism $\varphi_{i}$ for each $i, 1\leq i\leq n $. The image of $S_i$ (hence $T_i$) in $M$ will be denoted $\Sigma^0_i$. The fundamental group $\pi_1(M)$ of $M$ is a free group of rank $n$. We identify the free group $\mathbb F_n$ of rank $n$ with $\pi_1(M)$ up to conjugacy of the base point.

 \begin{definition}
A smooth embedded $2$-sphere $S$ in $M$ is said to be \emph{essential} if it does not bound a $3$-ball in $M$.
\end{definition}
We shall always consider essential smooth embedded $2$-spheres in $M$ throughout this paper. So by an embedded sphere, we mean an essential embedded $2$-sphere in $M$.

\begin{definition}
Two embedded $2$-spheres $S$ and $S'$ in $M$ are \emph{parallel} if they are isotopic, i.e., they bound a manifold of the form $S^2\times (0,1)$.
 \end{definition}

 By Laudenbach's work (~\cite{La1}), the embedded spheres $S$ and $S'$ in $M$ are parallel if and only if they are homotopic to each other.

\begin{definition}
\emph{A system of $2$-spheres} in $M$ is defined to be a finite collection of disjointly embedded essential $2$-spheres $S_i \subset  M$ such that no two spheres in this collection are parallel.
\end{definition}

A system of spheres in $M$ is \emph{maximal} if it is not properly contained in another system of spheres in $M$.

\begin{definition}
A \emph{reduced} sphere system is a sphere system in $M$ with connected, simply-connected complement, i.e., it cuts the manifold $M$ into a single simply-connected component.
\end{definition}

Note that a reduced sphere system in $M= \sharp_{n} S^{2}\times S^{1} $ contains exactly $n$ non-separating sphere components and it defines a basis of $H_2(M,\mathbb Z)$. Let $\Sigma=\Sigma_1\cup\dots\cup \Sigma_n$ be a reduced sphere system in $M$. We choose a transverse orientation for each sphere $\Sigma_i$ so we may speak of a positive and a negative side of $\Sigma_i$. Let $p\in M$ be a base point in the complement of $\Sigma$. A basis dual to $\Sigma$ is a set of homotopy classes of loops $\gamma_1,\dots,\gamma_n$  in $M$ based at $p$ such that the loop $\gamma_i$ is disjoint from $\Sigma_j$ for all $j\neq i$ and intersects $\Sigma_i$ in a single point. We orient $\gamma_i$ such that it approaches $\Sigma_i$ from the positive side. Since the complement of $\Sigma$ is simply connected, the homotopy classes of loops $\gamma_i$ define a basis of $\pi_1(M, p)$.  Thus, a  reduced sphere system gives a basis of the fundamental group of $M$. Note that the sphere system $\Sigma^0_i$ described in Subsection \ref{modelm} is a reduced sphere system. By splitting $M$ along a reduced sphere system $\Sigma=\cup_i \Sigma_i$, we get a manifold $M_{\Sigma}$ which is a $2n$-punctured $S^3$, i.e., $S^3$ with interiors of $2n$ disjointly embedded $3$-balls removed. Corresponding to each  sphere $\Sigma_i$, we get two boundary spheres $\Sigma^+_i$ and $\Sigma^-_i$ in $M_S$.

\subsection{The Mapping class group of $M$ and $Out(\mathbb F_n)$}

\begin{definition}
The \emph{mapping class group} $\mathcal{M}ap(M)$ of $M$ is the group of isotopy classes of orientation-preserving diffeomorphisms of $M$.
\end{definition}

 It is due to Laudenbach that we can replace diffeomorphisms by homeomorphisms in the definition of the mapping class group of $M$ (see \cite{La1}). As we have identified $\mathbb F_n$ with $\pi_1(M)$, we have a natural homomorphism
   $$\Phi: \mathcal{M}ap(M) \to Out(\mathbb F_n),$$
      where $Out(\mathbb F_n)$ is the group of outer automorphisms of $\mathbb F_n$. By Laudenbach's work (\cite{La2}), the above homomorphism $\Phi$ is surjective and has finite kernel generated by Dehn-twists along essential embedded $2$-spheres in $M$ .

\subsection{Sphere complex}\label{spc}
We shall recall the following definition:
 \begin{definition}
 The \emph{sphere complex} $\mathbb S(M)$ associated to $M$  is the simplicial complex whose vertices are isotopy classes of essential embedded spheres in $M$. A finite collection of isotopy classes of embedded spheres in $M$ is deemed to span a simplex in $\mathbb S(M)$ if they can be realized disjointly (up to isotopy) in $M$.
 \end{definition}

 The mapping class group $\mathcal{M}ap(M)$ of $M$ acts simplicially on the sphere complex $\mathbb S(M)$. This yields a homomorphism
 $$\Phi': \mathcal{M}ap(M) \to Aut(\mathbb S(M)),$$
  where $Aut(\mathbb S(M))$ is the group of simplicial automorphisms of the sphere complex $\mathbb S(M)$. It follows from the work of Laudenbach that kernel of $\Phi$ and $\Phi'$ are equal. In particular, the action of $\mathcal{M}ap(M)$ on the sphere complex $\mathbb S(M)$ induces a simplicial action of $Out(\mathbb F_n)$ on the sphere complex. This yields a homomorphism
  $$\Phi'': Out(\mathbb F_n)\to Aut(\mathbb S(M)).$$
    In ~\cite{AS}, Aramayona and Souto have shown that this homomorphism is an isomorphism.

\section{Embedded spheres in $M$}\label{splittingsembed}

Group theoretically, embedded spheres in $M$ correspond to splittings of the fundamental group of $M$.
 We recall the following definitions.

Let $G_1$ and $G_2$ be two groups with presentations $ G_1=\langle S_1\mid R_1\rangle$ and $ G_2=\langle S_2\mid R_2\rangle$. Let $f_1 : H \rightarrow G_1$ and $f_2 : H\rightarrow G_2$ be two homomorphisms from group $H$.

\begin{definition}
The \emph{free product} $G_1*G_2$ of $G_1$ and $G_2$ is the group with a presentation $ \langle S_1 \bigsqcup S_2 \mid R_1\bigsqcup R_2\rangle$, where $\bigsqcup$ denotes disjoint union.
\end{definition}

\begin{definition}
The \emph{amalgamated free product} $G_1*_H G_2 $ is defined as $(G_1 * G_2)/N $, where $N$ is the normal subgroup of $G_1* G_2$ generated by elements of the form $f_1(h)(f_2(h))^{-1}$ for $h \in H $.
\end{definition}

\begin{definition}
Let $G $ be a group with presentation $ G=\langle S\mid R\rangle$ and let $\alpha$ be an isomorphism between two subgroups $H$ and $K$ of $G$. Let $t$ be a new symbol not in $S$ and define
   $$ G*_{\alpha} = \langle S,t \mid R ,  tht^{-1}=\alpha(h), \forall h\in H\rangle.$$
The group $G *_\alpha$ is called the HNN-extension of $G$ relative to $\alpha$. The original group $G$ is called the base group for the construction and the subgroups $H$ and $K$ are the associated subgroups. The new generator $t$ is called the stable letter. Sometimes, we also write $G*_H$ for $G*_\alpha$.
\end{definition}

\begin{definition}
A \emph{splitting of a group} $G$ consists either of proper subgroups $A$ and $B$ of $G$ and a subgroup $H$ of $A \cap B$ with $A\neq H \neq B$ such that the natural map $A *_H B \rightarrow G$ is an isomorphism or it consists of a subgroup $A$ of $G$ and subgroups $H_0$ and $H_1$ of $A$ such that there is an element $t$ of $G$ which conjugates $H_0$ to $H_1 $ and the natural map $A*_H \rightarrow G$ is an isomorphism.
\end{definition}

Note that the condition that $A \neq H \neq B$ is needed as otherwise any group $G$ would split over any subgroup $H$. For, one can always write $G = G *_H H$.

  We shall see how an embedded sphere in $M$ corresponds to a splitting of $G=\pi_1(M)=\mathbb F_n$.

 We shall recall the description of $M$ given in the beginning of Section \ref{model}. Fix a base point $x_0$ away from $\Sigma^0i$. For each $1\leq i\leq n$, consider the element $\alpha_i \in \pi_1(M)$ represented by a closed path $\gamma_i$ in $M$ starting from $x_0$, going to $A_i$, piercing $\Sigma^0_i$ and returning to the base point from $B_i$. We choose this closed path $\gamma_i$ such that it does not intersect any $\Sigma^0_j$, $j\neq i$. Then, as the complement of $\Sigma^0_i$ in $M$ is simply-connected, the collection $\{\alpha_1,\dots,\alpha_n\}$ forms a free basis of $G=\pi_1(M)$. So, we have $G=\langle \alpha_1,\dots,\alpha_n\rangle$. Any directed closed path in $M$ hitting the $\Sigma^0_i$ transversely represents a word in $\{\alpha_1,\dots,\alpha_n\}$ by the way it pierces each $\Sigma^0_i$ and the order in which it does so. Without a base point chosen, such a closed path represents a conjugacy class or equivalently the cyclic word in $\mathbb F_n$. We call the basis $\{\alpha_1,\dots,\alpha_n\}$ as a standard basis and the spheres $\Sigma^0_1,\dots,\Sigma^0_n$ as standard basic spheres in $M$.

\begin{proposition}
Let $S$ be an essential embedded sphere in $M$. Then, $S$ gives a splitting of the fundamental group of $M$.
\end{proposition}
\begin{proof}
If $S$ separates $M$, then using Van Kampen's theorem, we get a splitting of the fundamental group $G$ of $M$ as a free product of its subgroups.

Now, suppose $S$ is non-separating. Choose a regular neighborhood $V = S \times [-1,1]$ of $S$ and an embedded path $\gamma$ in $M - V$ from a point of $S \times {-1}$ to a point $S \times {1}$. The sphere $S'$ which is the connected sum of $S \times {-1}$ with $S \times {1}$ along with the boundary of a regular neighborhood $U$ of $\gamma$, clearly bounds $U \cup V$ in $M$. We have $M = U \cup V \cup (M- (U\cup V))$. The set $U \cup V$ is a one punctured $(S \times S^1)$ with boundary $S'$ and $(M- (U\cup V))$ is a one punctured $3$-manifold $M'$ with boundary $S'$. Thus, $M$ is a connected sum of $S^2 \times S^1$ with the $3$-manifold $M'$. Then, we get a splitting of $G=G'* \langle t\rangle$, where $\pi_1(M')=G'$ and $\pi_1(S^2\times S^1)= \langle t\rangle$. Thus, $G$ can be viewed as an HNN-extension of $G'$ over the trivial subgroup $\{1\}$ of $G'$.
\end{proof}

 Note that when we apply Van Kampen's theorem, we choose a base point $y_0$ on the sphere $S$. If we choose a path $\gamma$ from  $x_0$ to $y_0$ in $M$, then we get a splitting, say $L$, of $G$ which consists of subgroups of $G$. Choosing another path $\gamma'$ from $x_0$ to $y_0$, we get a splitting of $G$ conjugate to $L$. So, the sphere $S$ gives a splitting of $G$ up to conjugacy.

Now, we shall see the converse,

\begin{proposition}\label{splittings}
Given a splitting of $G$, there exists an embedded sphere $S$ in $M$ which gives that splitting of $G$.
\end{proposition}

\begin{proof}
Suppose $G=F_1*F_2$. As subgroups of free group are free, both $F_1$ and $F_2$ are free. Choose free bases $\{a_1,\dots,a_m\}$ and $\{b_{m+1},\dots,b_{m+k}\}$ of $F_1$ and $F_2$, respectively. The set $ \{a_1,\dots,a_m,b_{m+1},\dots,b_{m+k}\}$ forms a free basis for $G$. Therefore, $m+k=n$. Any two bases of a free group
are equivalent in the sense that there exists an automorphism of that free group sending one basis to another. So, we have an automorphism $\phi$ of $G$ sending the basis $\{a_1,\dots,a_m,b_{m+1},\dots,b_n\}$ to the standard basis with $\phi(a_i)=\alpha_i$, for $1\leq i\leq m$ and $\phi(b_{m+j})=\alpha_{m+j}$, for $1\leq j\leq k$.

Every automorphism of a free group is finite composition of Nielsen automorphisms and every Nielsen automorphism of $G$ is induced by a homeomorphism of $M$ which fixes the base point $x_0$ (see ~\cite{LS},~\cite{St},~\cite{St3}). Thus, every automorphism of $G$ is induced by a homeomorphism of $M$ fixing the base point. Let $h$ be a homeomorphism of $M$ which fixes the base point and induces the automorphism $\phi$ on $G$. The element $\phi(a_i)=h_*(a_i)=\alpha_i$, for $1\leq i\leq m$, corresponds to the basic standard sphere $\Sigma^0_i$ and $\phi(b_{m+j})=h_*(b_{m+j})=\alpha_{m+j}$, for $1\leq j\leq k$ corresponds to the basic standard sphere $\Sigma^0_{m+j}$. We can choose an embedded sphere $S$, disjoint from all $\Sigma^0_i$ such that it partitions the collection of basic standard spheres into two sets, namely, $\{\Sigma^0_i,\dots,\Sigma^0_m\}$ and $\{\Sigma^0_{m+1},\dots,\Sigma^0_{m+k}\}$ and $x_0\in S$. Then, $S$ gives a free splitting of $G=A*B$, where $A=\langle \alpha_1,\dots,\alpha_m \rangle$ and $B=\langle \alpha_{m+1},\dots,\alpha_{m+k}\rangle$. Now, the sphere $h^{-1}(S)=S'$ gives partition of the collection of spheres $\{h^{-1}(\Sigma^0_1),\dots, h^{-1}(\Sigma^0_{m+k})\}$ into two sets $\{h^{-1}(\Sigma^0_i),\dots,h^{-1}(\Sigma^0_m)\}$ and $\{h^{-1}(\Sigma^0_{m+1}),\dots,h^{-1}(\Sigma^0_{m+k})\}$. The sphere structure $\{h^{-1}(\Sigma^0_1),\dots, h^{-1}(\Sigma^0_{m+k})\}$ corresponds to the basis $\{a_1,\dots,a_m,b_{m+1},\dots,b_{m+k}\}$.  Then, by applying Van-Kampen's theorem, we can see that $S'$ gives the splitting $G=F_1*F_2$ as $h(x_0)=x_0\in S'$.

If we choose a sphere $S$ such that the base point $x_0$ does not lie in $S$, then we choose a path $\gamma$ from $x_0$ to a point $y_o$ in $S$ such that $\gamma$ does not intersect any standard basic sphere. We choose curves $\alpha'_i$ starting from $y_0$, going to $A_i$, piercing $\Sigma^0_i$, and returning to the base point from $B_i$ without hitting any $\Sigma^0_j$, $j\neq i$. Note the homotopy class of closed curve the $\gamma*\alpha'_i*\bar{\gamma}$ gives the generator $\alpha_i$. The homotopy class of closed curve $h^{-1}(\gamma*\alpha'_i*\bar{\gamma})$ is $a_i$. From this, we can see that the sphere $S'$ gives the splitting $G=F_1*F_2$.

Note that if $G$ is an HNN-extension of a subgroup $G'$ of $G$ relative to the subgroups $H,K$ of $G'$ and an isomorphism $\theta :H\rightarrow K$, then $H=K=\{1\}$ and $G'$ is a subgroup of rank $n-1$ as $G$ is a free group of rank $n$. Thus, $G= G'*\langle t \rangle$, where $t\in G$.  We choose a basis $\{c_1,\dots,c_{n-1}\}$ of $G'$. The set $\{c_1,\dots,c_{n-1},t\}$ forms a basis of $G$. Then, we have an isomorphism $\phi'$ of $G$ sending the basis $\{c_1,\dots,c_{n-1},t\}$ to the standard basis with $\phi'(c_i)=\alpha_i$, for $1\leq i\leq n-1$ and $\phi'(t)=\alpha_n$. Let $h'$ be a homeomorphism of $M$ which fixes the base point and induces the automorphism $\phi'$ on $G$. The element $\phi'(c_i)=h'_*(c_i)=\alpha_i$, for $1\leq i\leq n-1$, corresponds to the basic standard sphere $\Sigma^0_i$ and $\phi'(t)=h'_*(t)=\alpha_n$ corresponds to the basic sphere $\Sigma^0_n$. The sphere structure $\{h'^{-1}(\Sigma^0_1),\dots,h'^{-1}(\Sigma^0_{n-1}),h'^{-1}(\Sigma^0_{n})\}$ corresponds to the basis $\{c_1,\dots,c_{n-1},t\}$ of $G$. Now, one can easily see that the sphere $h'^{-1}(\Sigma^0_n)$ gives a splitting of $G$ as an HNN-extension of $G'$ over the trivial subgroup $\{1\}$.
\end{proof}

\subsection{Isotopy classes of embedded spheres and conjugacy classes of corresponding splittings of $G$}
Now, we shall see that the isotopy classes of embedded spheres in $M$  correspond to the conjugacy classes of splittings of $G$.

We need the following: For a set $E \subset G$, we denote the complement of $E$ by $E^*$ and by $E^{(*)}$, we mean one of the sets $E$ and $E^*$.

\begin{definition}
Two subsets $E$ and $E'$ of the group $G$ are said to be \emph{almost equal} if their symmetric difference is finite.
\end{definition}
A set $E$ is said to be \emph{non-trivial} if both $E$ and $E^*$ are infinite.

\begin{definition}
A set $E \subset G$ is said to be \emph{almost invariant} if $E$ is almost equal to $Eg$ for all $g \in G$.
\end{definition}

 An equivalent condition in terms of the Cayley graph $\Gamma(G)$ is that the set $\delta E$ of edges of $\Gamma(G)$ with one vertex in $E$ and the other in
$E^*$ is finite. This is due to Cohen, ~\cite{CO}.
\begin{definition}
 Two almost invariant sets $E$ and $E'$ are said to be \emph{small} if $E\cap E'$ is finite.
 \end{definition}

A description of $\M$ related to the description of $M$ given in the beginning of the Section \ref{model} is as follows:  Let $\Gamma(G)$  be the Cayley graph of $G$ with respect to the standard basis $\{\alpha_1,\dots,\alpha_n\}$. Then, $\Gamma(G)$ is a tree with vertices $G$. The space $\widetilde{M}$ is obtained by taking a copy $gP$ of $P$ for each vertex $g$ of $\Gamma(G)$ and identifying $gS_i$ with $g\alpha_i T_i$ using $\varphi_i$, for each $g\in G$ and for each generator $\alpha_i$, $1 \leq i \leq n$. We observe that by the Hurewicz Theorem (~\cite{Ht1}), $\pi_2(M)=\pi_2(\widetilde{M})=H_2(\widetilde{M})$.

There is a natural embedding of $\Gamma(G)$ in $\M$ and in particular, $G$ can be identified with a subset of $\M$. Note that the set of ends of $\Gamma(G)$ can be identified with the set of ends of $\M$. Suppose that $\widetilde{S} \subset \M$ is an embedded sphere. As $\M$ is simply-connected, $\widetilde{S}$ separates $\M$ into two non-compact complementary components and hence gives a partition of the set $E(\M)$ of ends of $\M$ into two non-empty open subsets of $E(\M)$. Let $X_1$ and $X_2$ be the closures of the complementary components of $\widetilde{S}$. Let $E_i = X_i\cap G$. As only finitely many edges of the Cayley graph $\Gamma(G)$ intersect $\widetilde{S}$, the set $\delta E_1$ of edges of $\Gamma(G)$ with one vertex in $E_1$ and the other in $E^*_1$ is finite. This implies that the sets $E_1$ and $E_2$ form complementary almost invariant subsets of $G$.  The sets $E^{(*)}_ 1$ are called the almost invariant sets of $\widetilde{S}$.

Note that the sets $E(X_i)$ of ends of the sets $X_i$ are determined by the sets $E_i$. Hence, for two embedded spheres $\widetilde{S}$ and $\widetilde{S'}$, the corresponding almost invariant sets are almost equal if and only if $\widetilde{S}$ and $\widetilde{S'}$ are homologous as homology classes of embedded spheres in $\M$ are determined by the partitions of the set $E(\M)$ into two non-empty open sets, see ~\cite{GP1}.

An embedded sphere $S$ in $M$ lifts to a collection of embedded spheres in $\M$. The collection of almost invariant sets in $G$ of the lifts of $S$ to $\M$ is called the collection of almost invariant sets of the splitting of $G$ corresponding to $S$.

Consider the lifts $\widetilde{S}$ and $\widetilde{S'}$ to $\M$ of embedded spheres $S$ and $S'$ in $M$. Let $E^{(*)}$ and ${E'}^{(*)}$ be almost invariant sets associated to $\widetilde{S}$ and $\widetilde{S'}$ respectively. If $\widetilde{S}$ and $\widetilde{S'}$ are isotopic, then their almost invariant sets are almost equal. Therefore, two of the four sets $E^{(*)}\cap {E'}^{(*)}$ are small. Then, by ~\cite[Lemma 2.3]{SS}, the two splittings of $G$ given by $S$ and $S'$ are conjugate. For more details of the splittings of groups and their almost invariant sets, refer ~\cite{SS}.

Conversely, suppose that the splitting of $G$ associated to $S'$ is conjugate of the splitting of $G$ associated to $S$ by an element $g\in G$. Let $E^{(*)}$ be the almost invariant sets associated to some lift $\widetilde{S}$ of $S$ to $\M$. Then, the sets $gEg^{-1}$ and $gE^*g^{-1}$ are the almost invariant sets associated to some lift $\widetilde{S'}$ of $S'$ to $\M$. As $E$ (respectively, $E^*$) is almost invariant subset of $G$, $Eg^{-1}$ (respectively, $E^*g^{-1}$) is almost equal to $E$ (respectively, $E^*$). This implies that $gE^{(*)}g^{-1}$ is almost equal to $gE^{(*)}$ which are almost invariant sets associated to the translate $g\widetilde{S}$ of $S$. Therefore, $\widetilde{S'}$ and $g\widetilde{S}$ are homologous and hence isotopic. This implies $S$ and $S'$ are homotopic and hence isotopic by Laudenbach's theorem.

Thus, in group theoretic terms, isotopy classes of embedded spheres in $M$ correspond to conjugacy classes of splittings of the free group $G=\pi_1(M)$.

\section{ Graph of groups decomposition of $G=\pi_1(M)$}\label{GOG}

\subsection{Graph of groups}

Recall the following:
\begin{definition}\label{graph} A \emph{graph} $\mathbb T$ consists of two sets $E(\mathbb T)$ and $V(\mathbb T)$, called the edges and vertices of $\mathbb T$, a mapping from $E(\mathbb T)$ to $E(\mathbb T)$ denoted $e \mapsto \bar e$, for which $e\neq \bar e$ and $\bar{\bar e} = e$ and a mapping from $E(\mathbb T)$ to $V(\mathbb T) \times V(\mathbb T)$, denoted $ e \mapsto (o(e),t(e))$, such that $\bar e \mapsto (t(e),o(e))$ for every $e \in E(\mathbb T)$.
\end{definition}

A graph $\mathbb T$ has an obvious geometric realization $|\mathbb
T|$ with vertices $V(\mathbb T)$ and edges corresponding to pairs
$(e,\bar{e})$. When we say that $\mathbb T$ is connected or has
some topological property, we shall mean that the realization of
$\mathbb T$ has the appropriate property.

\begin{definition}
A \emph{graph of groups} $\mathcal G$ consists of a graph $\Gamma$ together with a function assigning to each vertex $v$ of $\Gamma$ a group $G_v$ and to each edge $e$ a group $G_e$, with $G_{\bar{e}} =G_e$ and an injective homomorphism $f_e: G_e \rightarrow G_{\del_0 e}$.
\end{definition}

\subsection{Fundamental group of a graph of groups}\label{funda}

We shall see an algebraic definition of the fundamental group of graph of group.

First, choose a spanning tree $\mathbb T$ in $\Gamma$.
\begin{definition}
The \emph{fundamental group of $\mathcal G$} with respect to $\mathbb T$, denoted $\mathcal G_\Gamma(\mathbb T)$, is defined as the quotient of the free product  $(\ast_{v\in \Gamma(V)} (G_v) \ast F(E)$, where $F(E)$ is a free group with free basis $E=Edge(\Gamma)$, subject to the following relations:
\begin{enumerate}
\item $ \overline{e} f_e(g)e=f_{\overline{e}}(g)$, for every $e$ in $E$ and every $g\in G_e$.

\item $ \overline e e=1$, for every $e\in E$.

\item $e = 1$, for every edge $e$ of the spanning tree $\mathbb T$.

\end{enumerate}
\end{definition}

One can see that the above definition is independent of the choice of the spanning tree $\mathbb T$ of $\Gamma$.

Observe that in the case when $\Gamma$ has just one pair $(e,\bar e)$ of edges and two vertices  $v_1$ and $v_2$ and  if groups associated to $v_1$, $v_2$ and $(e,\bar e)$ are $A$, $B$ and $C$ respectively, the fundamental group $\pi_1(\mathcal G)$ is $A*_C B$. In the case, when $\Gamma$ has just one pair $(e,\bar e)$ of edges and one vertex $v$ and if the associated groups are $C$ and $A$ respectively, then the fundamental group $\pi_1(\mathcal G)$ is $A*_C$.

We call a graph of groups $\mathcal G$ with the fundamental group equal to $G=\pi_1(M)$ a graph of groups structure of $G$ or graph of groups decomposition of $G$. A splitting of a group $G=\pi_1(M)$ is a graph of groups $\mathcal G$ with a connected graph $\Gamma$ containing only one edge such that the fundamental group $\mathcal G_\Gamma$ is $G$.  Note that for a graph of groups structure of $G$, edge groups are trivial. A graph of groups structure of $G$  on a graph with one edge and two vertices corresponds to a free product of vertex groups. A graph of groups structure of $G$  on a single vertex with a loop corresponds to an HNN- extension.

Let $\mathcal G$ be a graph of groups structure of $G$ with the underlying graph $\Gamma$ containing $k$ edges. We number the edges as $1,2,\dots, k$. We call the spiltting of $G$ obtained by collapsing all the edges except the $i$-th edge as $i$-th elementary splitting of $\mathcal G$ (or elementary splitting associated to $i$-th edge of $\Gamma$).

Now, we shall show that given any graph of groups structure $\mathcal G$ of $G=\pi_1(M)$ up to conjugacy of vertex groups, there is a corresponding isotopy class of a sphere system in $M$. We shall recall the following:

\subsection{Graph of topological spaces}
We define a \emph{graph $\chi$ of topological spaces} or of spaces with preferred base point as we defined graph of groups: here, it is not necessary for the map $X_e \rightarrow X_{\del_0 e}$ to be injective, as we can use the mapping cylinder construction to replace the maps by inclusions and this does not alter the total
space defined below.  Given a graph $\chi$ of spaces, we can define total space $\chi_{\Gamma}$ as the quotient of
$$\cup \{X_v : v\in V(\Gamma)\}  \cup \{\cup \{X_e \times I :e \in E(\Gamma)\}\} $$
by  identifications,
$$X_e \times I \rightarrow X_{\bar e} \times I \text{  by  }  (x,t) \rightarrow (x,1-t)$$

$$X_e \time 0 \rightarrow X_{\del_0 e} \text{  by  }   (x,0) \rightarrow f_e(x).$$

If $\chi$ is a graph of (connected) based spaces, then by taking fundamental groups we obtain a graph of groups $\mathcal G$ with the same underlying  graph $\Gamma$. The fundamental group $\pi_1(\mathcal G)$ of the graph of groups  $\mathcal G$ is defined to be the fundamental group of the total space $\chi_\Gamma$.

Note that the fundamental group of the graph of groups  $\mathcal G$ defined in Subsection ~\ref{funda} is the same (isomorphic) as defined here.  So, we now denote $\pi_1(\mathcal G)$ by $\mathcal G_\Gamma$. One can show that $\mathcal G_\Gamma$ is independent of the choice of $\chi$. Note that each map $G_v\to \mathcal G_\Gamma$ is injective.

\subsection{Graph of groups structure of $\pi_1(M)$ associated to a sphere system}
Given a sphere system $S=\cup_i S_i$ in $M$, we can associate a graph of groups $\mathcal G$ to the sphere system  $S$ with underlying graph $\Gamma$ such that the fundamental group $\mathcal G_\Gamma$ of $\mathcal G$ is $G=\pi_1(M)$ as follows:  The vertices of $\Gamma$ are the closures of components of $M - S$ and edges are the spheres $S_i$. An edge $e$ in $\Gamma$ is adjacent to $v$ if the sphere corresponding to $e$ is a boundary component of the component $X_v$ corresponding $v$ and the maps from the edges to vertices are inclusion maps of spheres. This gives the graph $\chi$ of topological spaces  with the same underlying graph $\Gamma$. We can consider based spaces by fixing a base point in each space which is a complementary component of $M-S$. Then, by taking fundamental groups of the spaces associated to vertices and edges of $\Gamma$, we get a graph of groups $\mathcal G$ with the underlying graph $\Gamma$. The edge groups $G_e$ are trivial. We can see that the total space $\chi_{\Gamma}$ is $M$ and $\mathcal G_\Gamma =\pi_1(M)$. The fundamental group $G_v$ of $X_v$ can be viewed as subgroup (up to conjugacy) of $G$ as the map $G_v\to G$ is injective. Note that $\Gamma$ has no terminal vertex with trivial vertex group and no two elementary splittings of $\mathcal G$ are conjugates of each other as sphere components of $S$ are essential and no two spheres components are isotopic.

Now, we prove the following theorem:

\begin{theorem}\label{graphsphere}
 Let $\mathcal G$ be a graph of groups structure of $G=\pi_1(M)$ with underlying graph $\Gamma$ such that $\Gamma$ has no terminal vertex with trivial vertex group associated to it and no two elementary splittings of $\mathcal G$ are conjugates of each other. Then, there exist a system of embedded spheres in $M$ such that the graph of groups structure of $G$  given by this sphere system is $\mathcal G$ (up to conjugacy of vertex groups).
 \end{theorem}

 \begin{proof}
Note that as $G$ is a free group and $G$ is the fundamental group of $\mathcal G$, each edge group must be trivial and all the groups associated to vertices are free groups (subgroups of $G$).

First, choose a spanning tree $\mathbb T$ in $\Gamma$. The fundamental group $G=\mathcal G_\Gamma$ of $\mathcal G$ with respect to $\mathbb T$, is $\ast_{v\in \Gamma(V)} (G_v) \ast F(E')$, where $F(E')$ is a free group with free basis $E'$ which contains edges of $\Gamma$ which are not in the spanning tree $\mathbb T$. Choose a basis $A_v=\{a^v_1,\dots,a^v_{n_v}\}$ of the subgroup $G_v$ of $G$ associated to vertex $v$. Let $B=\{b_1,\dots,b_{n'}\}$ be a free basis of $F(E')$. The set $C=\cup_v A_v \cup B$ forms a free basis of $G=\pi_1(M)$. Suppose $\Gamma$ has $l$ vertices $v_1,\dots,v_l$. We write the set $C$ as $\{ a^{v_1}_1\dots
a^{v_1}_{n_{v_1}},a^{v_2}_{n_{v_1}+1},\dots,a^{v_2}_{n_{v_1}+n_{v_2}},\dots ,a^{v_l}_{n{v_1}+\dots+n_{v_l}},b_{n{v_1}+\dots+n_{v_l}+1},\dots,b_n\}$.
We have an automorphism $\phi$ of $G$ sending the basis $C$ to the standard basis with $\phi(a^{v_1}_i)=\alpha_i$, for $1\leq i\leq n_{v_1}$, $\phi(a^{v_2}_i)=\alpha_i$, for $n_{v_1}< i\leq n_{v_1}+n_{v_2}$, \dots, $\phi(a^{v_l}_i)=\alpha_i$, for $n_{v_1}+\dots+n_{v_{l-1}}< i\leq n_{v_1}+\dots+n_{v_l}$ and $\phi(b_i)=\alpha_i$, for $n_{v_1}+\dots+n_{v_l}< i\leq n$. Every automorphism of $G$ is induced by a homeomorphism of $M$ fixing the base point. Let $h$ be a homeomorphism of $M$ which fixes the base point and induces the automorphism $\phi$ on $G$.

For the vertex $v_i$, let $e^{v_i}_1,\dots,e^{v_i}_{k_{v_i}}$ be the edges adjacent to $v_i$ such that the edge $e^{v_i}_j$ is not a loop bases at $v_i$. Let $e'^{v_i}_1,\dots,e'^{v_i}_{k'_{v_i}}$ be the edges which are loops based at $v_i$. As the vertex group $G_{v_i}$ is a free group of rank $n_{v_i}$, we consider the space $N_i$ as connected sum of $n_{v_i}$ copies of $S^2\times S^1$ with fixed a basic (reduced) sphere system $\mathcal B_i=\{ B^{v_i}_1,\dots,B^{v_i}_{n_{v_i}}\}$ in each space $N_i$. We remove interiors of $k_{v_i}+2k'_{v_i}$  disjointly embedded $3$-balls from $N_i$ such that these balls are disjoint from the basic sphere system $\mathcal B_i$ of $N_i$. To each edge $e^{v_i}_j$, we label exactly one boundary spheres by $S^{v_i}_j$ and for each loop $e'^{v_i}_k$, we label exactly two boundary sphere $S'^{v_i}_k$. We denote this new space by $N'_i$. If the vertex group is trivial then we consider $N_i$ as the $3$-sphere $S^3$ and then construct space $N'_i$ as described above.

Now, we identify the boundary spheres of all $N'_i$'s which correspond to the same edge in $\Gamma$ using diffeomorphisms with appropriate orientation preserving or reversing property so that the resultant space $M'$ is homeomorphic to $M$ and images in $M$ of the spheres corresponding to edges of $\Gamma$ give the graph of topological spaces structure of $M$ with the underlying graph $\Gamma$.

 Let $B_{n{v_1}+\dots +n_{v_l}+1},\dots,B_n$ be the spheres corresponding to the edges not in the spanning tree.  Note that the set $\mathbb B$ which is a union of all $\mathcal B_i$ together with the spheres in $M'$ which correspond to the edges not in the spanning tree $\mathbb T$, forms a reduced sphere system in $M'$.

Let $h_1$ be a homeomorphism from $M'$ to $M$ which maps a fixed base point $y_o$ in $M'$ to the base point $x_0$. We can assume that $y_0$ is away from all the sphere components of $\mathbb B$. Consider the $h_1$ images of the spheres in $M'$ corresponding to the edges in $\Gamma$. There exist a homeomorphism $h_2:M\to M$ such that $h_2(h_1(B^{v_1}_i))=\Sigma^0_i$, for $1\leq i\leq n_{v_1}$, $h_2(h_1(B^{v_2}_i))=\Sigma^0_i$, for $n_{v_1}< i\leq n_{v_1}+n_{v_2}$, \dots, $h_2(h_1(B^{v_l}_i))=\Sigma^0_i$, for $n_{v_1}+\dots+n_{v_{l-1}}< i\leq n_{v_1}+\dots+n_{v_l}$ and $h_2(h_1(B_i))=\Sigma^0_i$, for $n_{v_1}+\dots+n_{v_l}< i\leq n$. Consider the $h_2$ images of $h_1$ spheres. We call them as $h_2$-spheres. Then, we can see that $h^{-1}$ images of $h_2$-spheres give the decomposition of $M$ such that we can associate the graph of groups structure $\mathcal G$  to $G$ (up to conjugacy of vertex groups). As $\Gamma$ has no terminal vertex with trivial vertex group and no two elementary splittings of $\mathcal G$ are conjugates of each other, $h^{-1}$ images of $h_2$-spheres  is a sphere system in $M$.
\end{proof}

\section{Complex of non-separating spheres}\label{CNS}
In this section, we shall discuss the complex of non-separating embedded spheres in $M$.

\begin{definition}\label{nsc}
 The complex $NS(M)$ of \emph{non-separating} spheres in $M$ is the simplicial complex whose vertices are isotopy classes of non-separating embedded spheres in $M$. A finite collection of isotopy classes of non-separating embedded spheres in $M$ is deemed to span a simplex in $NS(M)$ if they can be realized disjointly (up to isotopy) in $M$.
\end{definition}

 Clearly, we have a natural inclusion of $NS(M)$ into the sphere complex $\mathbb S(M)$, i.e., $ NS(M)$ is a subcomplex of $\mathbb S(M)$. It is a subcomplex of the sphere complex spanned by the vertices which are isotopy classes of non-separating embedded spheres in $M$.

\subsection{Path-connectedness of $NS(M)$}
The $1$-skeleton of the complex $NS(M)$ is path-connected. For, let $A=[S]$ and $B=[S']$ be two vertices in $NS(M)$. Let $A=A_1,\dots,A_n=B$ be a shortest path in the sphere graph of $M$, i.e., the $1$-skeleton of the sphere complex $\mathbb S(M)$. We can choose each $A_i$ to be an isotopy class of a non-separating embedded sphere in $M$ as follows: Suppose that $A_i$ is an isotopy class of a separating embedded sphere $S_i$ in $M$ and let $M_1$ and $M_2$ be its two complementary components of $S_i$ in $M$. Suppose $A_{i-1}=[S_{i-1}]$ and $A_{i+1}=[S_{i+1}]$ such that both $S_{i-1}$ and $S_{i+1}$ are disjoint from $S_i$. If $S_{i-1}$ and $S_{i+1}$ are contained in different complementary components of $S_i$, then they are disjoint and hence $A_{i-1}$ and $A_{i+1}$ connected by an edge in the sphere graph. In this case, the vertex $A_i$  can be removed from the edge-path. If $S_{i-1}$ and $S_{i+1}$ are contained in the same component, say $M_1$, then one can replace $A_i$ by a isotopy class $A'_i$ of a non-separating embedded sphere $S'_i$ contained in $M_2$.

\subsection{Dimension and the diameter of $NS(M)$}
One can see that the dimension of the complex $NS(M)$ is $3n-4$ as there are maximal system $S=\cup_i S_i$ of $2$-spheres in $M$ such that each $S_i$ is non-separating in $M$. The diameter of the $1$-skeleton of the complex $NS(M)$ is infinite. For, if the diameter of the $1$-skeleton of $NS(M)$ is finite, then it will imply that the diameter of sphere graph is finite as every vertex in the sphere graph corresponding to an isotopy class of a separating embedded sphere in $M$ is at a distance one from a vertex in the sphere graph corresponding to an isotopy class of a non-separating embedded sphere in $M$. But this is not true as the diameter of the sphere graph is infinite (see ~\cite{KL}).

\subsection{The complex $NS(M)$ is Flag}

\begin{definition}
 A simplicial complex $\mathbb K$ is flag when, in dimensions two and higher, a simplex is present if its faces are, i.e., every complete subgraph
on $r + 1$ vertices contained in the $1$-skeleton is the $1$-skeleton of an $r$-simplex.
\end{definition}

Note that every automorphism of the $1$-skeleton of a flag simplicial complex extends uniquely to an automorphism of the whole complex. By \cite[Theorem 3.3]{GP}, one can see that the complex $NS(M)$ is a flag complex. Thus, every automorphism of the $1$-skeleton of the complex $NS(M)$ extends uniquely to an automorphism of $NS(M)$.

\subsection{Action of $Out(\mathbb F_n)$ on $NS(M)$}
One can see that a homeomorphism of $M$ maps a non-separating embedded sphere in $M$ to a non-separating embedded sphere in $M$. This defines a simplicial  action of the mapping class group $\mathcal{M}ap(M)$ on the complex $NS(M)$ and hence a simplicial action of the group $Out(\mathbb F_n)$ of outer automorphisms of $F_n$. In fact, $NS(M)$ is a subcomplex of the sphere complex invariant under the action of $Out(\mathbb F_n)$. Hence, $NS(M)$ is invariant under the action of $Aut(\mathbb S(M))$. The action of $Out(\mathbb F_n)$ on $NS(M)$ defines a homomorphism
$$\Upsilon :Out(\mathbb F_n)\to Aut(NS(M)),$$
 where $Aut(NS(M))$ is the group of simplicial automorphisms of $NS(M)$.  We shall show that this is an isomorphism in Section \ref{Main}.

We recall the following:

\begin{definition}
A \emph{simple} sphere system $\Sigma=\cup_i \Sigma_i$ in $M$ is a sphere system  in $M$ such each complementary component of $\Sigma$ in $M$ is simply-connected.
\end{definition}
A maximal sphere system in $M$ is simple. Splitting $M$ along a maximal sphere system produces a finite collection of $3$-punctured $3$-spheres. Here, a $3$-punctured $3$-sphere is the complement of the interiors of three disjointly embedded $3$-balls in a $3$-sphere.

  We make a note of the following known facts:

 \begin{itemize}
 \item  Every non-separating embedded sphere in $M$ is contained in a maximal system $\Sigma$ of $2$-spheres such that each component of $\Sigma$ is a non-separating sphere in $M$.

  \item  If $S$ is a non-separating component of a maximal system $\Sigma$ of spheres in $M$, then there is a reduced sphere system $\Sigma' \subset \Sigma$ with $S\in \Sigma'$. In particular, every non-separating embedded sphere in $M$ can be extended to a reduced sphere system in $M$.
 \end{itemize}

\section{The groups $Aut(NS(M))$ and $Out(\mathbb F_n)$}\label{Main}
In this section, we shall prove the main result of this paper.

\begin{theorem}\label{MainT}
The group $Aut(NS(M))$ of simplicial automorphisms of the complex $NS(M)$ is isomorphic to the group $Out(\mathbb F_n)$.
\end{theorem}

We call a simplex in $NS(M)$ \emph{simple} (respectively, \emph{reduced}) if it corresponds to a simple (respectively, reduced) sphere system in $M$.

Let  $\phi$ be an automorphism of $NS(M)$. Firstly, we shall prove the following lemma:

\begin{lemma}\label{mainlemma}
The automorphism $\phi$ maps a reduced simplex in $NS(M)$ to a reduced  simplex in $NS(M)$.
\end{lemma}

To prove the above lemma, we recall the following definition:

\begin{definition}
The \emph{link} of a simplex $\sigma$ in a simplicial complex $\mathbb K$ is a subcomplex of $\mathbb K$ consisting of the simplices $\tau$ that are disjoint from $\sigma$ and such that both $\sigma$ and $\tau$ are faces of some higher-dimensional simplex in $\mathbb K$.
\end{definition}

So if $\sigma$ is a simplex in $NS(M)$, $\sigma$ corresponds to a system $S=\cup_i S_i$ of non-separating embedded spheres in $M$. The $link(\sigma)$ is the subcomplex of $NS(M)$ spanned by the isotopy classes $[T]$ of non-separating embedded spheres $T$ in $M$ such that all $T$'s are disjoint from all and distinct from all $S_i$ up to isotopy. We shall look at the links of  reduced simplices and non-reduced simplices in $NS(M)$.

\subsection{Link of a reduced simplex}
Consider a reduced simplex $\sigma$ in $NS(M)$ corresponding to a reduced sphere system $\Sigma=\cup_i\Sigma_i$ in $M$. Cutting $M$ along $\Sigma$, then gives a simply-connected $3$-manifold $M'$ with $2n$ boundary components $\Sigma^+_i,\Sigma^-_i$, for $1\leq i\leq n$, where $\Sigma^+_i$ and $\Sigma^-_i$ correspond to the sphere $\Sigma_i$ for each $i$. Note that  $M'$ is homeomorphic to a $3$-sphere with interiors of $2n$ disjointly embedded $3$-balls removed. Now, every non-boundary parallel embedded sphere in $M'$ gives a partition of the boundary spheres of $M'$ into two sets, say $X_S$ and $Y_S$. Moreover, its isotopy class in $M'$ is determined by such a partition of the boundary spheres of $M'$.

 Consider an embedded sphere $S$ in $M'$ such that the partition $\{X_S,Y_S\}$ of the boundary spheres of $M'$ has the following property: There is a $\Sigma^{\epsilon}_i \in X_S$  for some sign $\epsilon\in\{+,-\}$ and for some $i$ such that $\Sigma^{-\epsilon}_i\in Y_S$. Such a sphere in $M'$ corresponds to a non-separating embedded sphere in $M$ which is disjoint from all and distinct from all the spheres $\Sigma_i$. Note that any embedded sphere in $M$ disjoint and distinct from all $\Sigma_i$ which corresponds to a partition of the boundary spheres of $M'$ with the above property is non-separating in $M$. The isotopy classes of such non-separating embedded spheres in $M$ span the link of the simplex $\sigma$ in $NS(M)$. One can easily see that the $link(\sigma)$ has a finite number of vertices.

\subsection{Link of an $(n-1)$-simplex which is not reduced}
 Consider an $(n-1)$-simplex $\sigma'$ in $NS(M)$ such that it corresponds to a  system $\Sigma'=\cup_i\Sigma'_i$ of non-separating spheres in $M$ which is not reduced.

 \begin{lemma}\label{notsimple}
 The sphere system $\Sigma'$ is not a simple sphere system.
 \end{lemma}

 \begin{proof}
   If $\Sigma'$ is simple, then it must have more than one complementary component as $\Sigma'$ is not a reduced sphere system. Consider the  graph of groups $\mathcal G$ with underlying graph $\Gamma$ associated to $\Sigma'$. The edges of $\Gamma$ correspond to the components of $\Sigma'$ and vertices correspond to the complementary components of $\Sigma'$ in $M$. Then, $\Gamma$ has more than one vertex with all the vertex groups and edge groups trivial. In this case, $F_n=\pi_1(M)=\pi_1(\mathcal G)$ is the fundamental group of $\Gamma$. But $\Gamma$ has more than one vertex and $n$ edges, so $\pi_1(\Gamma)$ can not be a free group on $n$ generators. Therefore, $\Sigma'$ can not be simple.
    \end{proof}

    \begin{figure}
\includegraphics[scale=0.4]{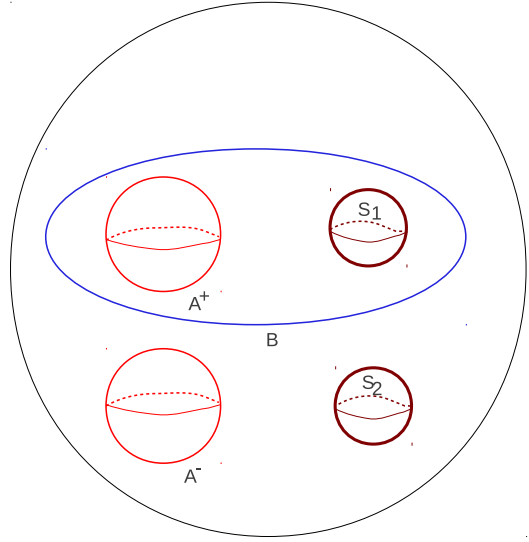}
\caption{Sphere $B$ in $P$=$2$-punctured $S^2\times S^1$}\label{fig64}
\end{figure}

\begin{figure}
\includegraphics[scale=0.4]{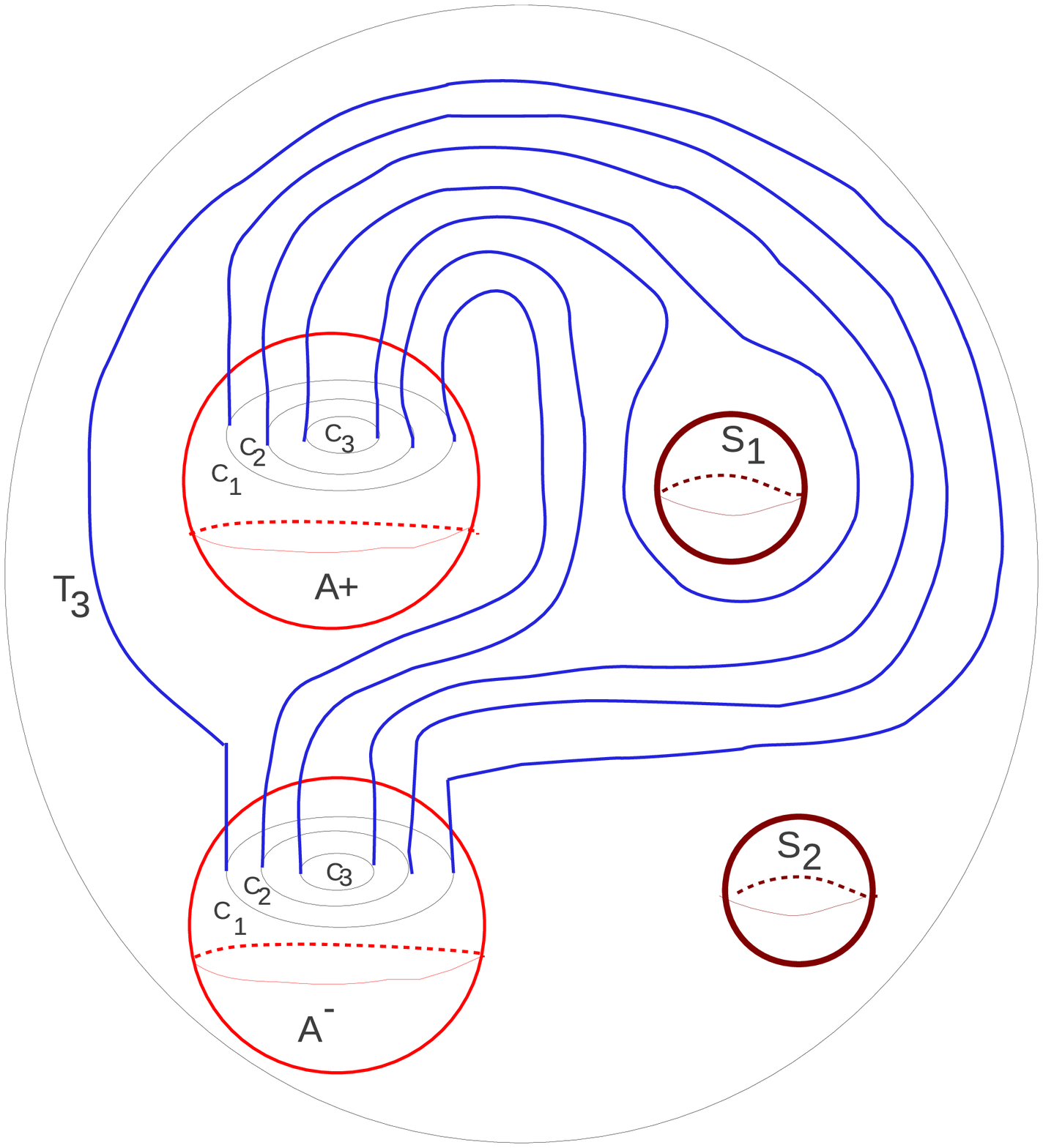}
\caption{Sphere $T_3$ in $P$=$2$-punctured $S^2\times S^1$}\label{fig67}
\end{figure}

      Cutting $M$ along $\Sigma'$, then gives a (not necessarily connected) $3$-manifold $M'$ with $2n$ boundary components $\Sigma'^+_i, \Sigma'^-_i$ for $1\leq i\leq n$ such that $\Sigma'^+_i$ and $\Sigma'^-_i$ correspond to the sphere $\Sigma'_i$ for each $i$. Now, as $\Sigma'$ is not reduced and each $\Sigma'_i$ is non-separating, by Lemma \ref{notsimple} there is at least one non simply-connected component of $M'$, say $R$, with at least two boundary spheres.  Note that there are at least two boundary spheres of $R$ which correspond to two distinct components of $\Sigma'$ and if $M'$ has more that one component, then there are at least two boundary spheres, say $S_1=\Sigma'^{\epsilon}_i$  and $S_2=\Sigma'^{\eta}_j$, for some $i$ and $j$ and for some signs $\epsilon,\eta \in\{+,-\}$ such that $\Sigma'^{-\epsilon}_i$ and $\Sigma'^{-\eta}_j$ are not  boundary spheres of $R$, i.e., they are boundary components of some other component of $M'$. Note that $R$ is a $p$-punctured $\sharp_{k} S^{2}\times S^{1}$, for appropriate integers $p\geq 2$ and $k\geq 1$. Now, we shall see that for $k=1$ and $p=2$, $link(\sigma')$ has infinitely many vertices by constructing infinitely many non- boundary parallel non-separating embedded spheres in $R$. The same construction will work for all the other  possibilities of $k$ and $p$ obviously.

   Now, suppose $k=1$ and $p=2$. In this case, $R$ is a 2-punctured $S^2\times S^1$. Let $S_1$ and $S_2$ be the two boundary components of $R$. These boundary spheres correspond to two distinct components of $\Sigma'$. In this case, $R$ has infinitely many embedded spheres which separate the boundary spheres $S_1$ and $S_2$. These spheres in $R$  are non-separating embedded spheres in $M$. We construct such examples as follows: Let $A$ be a non-boundary parallel, non-separating embedded sphere in $R$ disjoint from $S_1$ and $S_2$. If we cut $R$ along $A$, then we get a manifold $R'$ which is a $4$-punctured $S^3$ with boundary components $S_1, S_2, A^+$ and $A^-$, where $A^+$ and $A^-$ correspond to the embedded sphere $A$ in $R$. Consider an embedded sphere $B$ in $R'$ which separates $A^+$ and $S_1$ from $A^-$ and $S_2$ as shown in Fig. \ref{fig64}.

 The (image under identification map from $R'$ to $R$ of) embedded sphere $B$ is non-separating in $R$. Let $S'_1$ be an embedded sphere parallel to $S_1$ in $R$ which is disjoint from all the spheres $S_1,S_2$ and $A$. Now, we consider an embedded sphere $T_m$ which is connected sum of $B$ and $S'_1$ in $R$ along with a tube $t_m$ such that \begin{enumerate}
      \item the tube $t_m$ intersects transversely $m$ times the embedded sphere $A$ as follows: Starting from $B$ it hits $A^-$, piercing $A$, again hitting $A^-$ through $A^+$, it hits $(m-1)$ times the sphere $A$ in this fashion, it ends at $S'_1$ as shown in Fig. \ref{fig67}, for $m=3$.
       \item  The circles of intersection of $t_m$ with $A$ split $t_m$ into components called pieces such that each piece separates the boundary spheres $S_1$ and $S_2$.
        \end{enumerate}
        One can see that each sphere $T_m$ is non-separating in $M$ as it separates the boundary spheres of $R$. Each $T_m$ is disjoint from all and distinct from all the spheres $\Sigma'_i$. Thus, $link(\sigma')$ is infinite.

    \subsection{Automorphisms of $NS(M)$ and simple simplices}

    Firstly we shall see the proof of Lemma \ref{mainlemma}
    \begin{proof}[Proof of Lemma \ref{mainlemma}]
     If an automorphism of $NS(M)$ maps a simplex $\tau$ of $NS(M)$ to another simplex $\tau'$ of $NS(M)$, then it maps $link(\tau)$ to $link(\tau')$ isomorphically. As $link(\sigma)$ of a reduced simplex $\sigma$ has finitely many vertices and $link(\sigma')$ of a non-reduced $(n-1)$-simplex has infinitely many vertices, the automorphism $\phi$ can not map $link(\sigma)$ to $link(\sigma')$ and hence it can not map the reduced simplex $\sigma$ to the non-reduced simplex $\sigma'$. This shows that $\phi$ maps a reduced simplex in $NS(M)$ to a reduced simplex in $NS(M)$.
     \end{proof}

    \begin{lemma}\label{simplelemma}
    Every automorphism $\phi$ of $NS(M)$ maps a simple simplex in $NS(M)$to simple simplex in $NS(M)$.
    \end{lemma}
    \begin{proof}
    Note that a sphere system in $M$ which contains a reduced sphere system as its subset is simple. Moreover, every simple sphere system $\Sigma=\cup_i \Sigma_i$ in $M$ contains a reduced sphere system in $M$ as its subset. For, consider the  graph of groups $\mathcal G$ with underlying graph $\Gamma$ associated to $\Sigma$. The edges of $\Gamma$ correspond to the embedded spheres $\Sigma_i$  and the vertices correspond to the complementary components of $\Sigma$ in $M$. As $\Sigma$ is simple, $\Gamma$ has with all the vertex groups and edge groups trivial. In this case, $F_n=\pi_1(M)=\pi_1(\mathcal G)$ is the fundamental group of $\Gamma$. Consider a maximal tree $\mathbb T$ in $\Gamma$. Then, $\Gamma\setminus \mathbb T$ has $n$ edges, say $e_1,\dots,e_n$. Let $\Sigma_{i_1},\dots,\Sigma_{i_n}$ be the components of $\Sigma$ corresponding to the edges $e_1,\dots,e_n$. As the the tree $\mathbb T$ is connected and contains all the vertices of $\Gamma$, one can see that the complement of the embedded spheres $\Sigma_{i_k}$ in $M$ is connected. The fundamental group of the complement of these spheres $\Sigma_{i_k}$ in $M$ is the fundamental group of the graph of groups $\mathcal G'$ with underlying graph $\mathbb T$.  As $\mathbb T$ is a tree with all vertex groups and edge groups trivial, the fundamental group of $\mathcal G'$ is trivial. Hence, the complement of the spheres $\Sigma_{i_k}$ in $M$ is simply-connected. This shows that the $n$ embedded spheres $\Sigma_{i_k}$'s corresponding to these $n$ edges outside the maximal tree $\mathbb T$ form a reduced sphere system in $M$.

      If $\sigma$ is a simple simplex in $NS(M)$, it has a face $\tau$ which is reduced.  Then, $\phi(\sigma)$ is a simplex in $NS(M)$ with face $\phi(\tau)$. By Lemma \ref{mainlemma}, $\phi(\tau)$ is a reduced simplex  which implies that $\phi(\sigma)$ is a simple simplex in $NS(M)$. Thus, $\phi$ maps a simple simplex in $NS(M)$ to a simple simplex in $NS(M)$.
    \end{proof}

\subsection{Culler-Vogtmann Space and the reduced outer space}
   We briefly recall the definition of \emph{Culler-Vogtmann space} $CV_n$ for free group $\mathbb F_n$ of rank $n$ (\cite{CV},~\cite{VK}). It is also known as \emph{outer space}. A point in $CV_n$ is an equivalence class of marked metric graphs $(h,X)$ such that
   \begin{enumerate}
    \item $X$ is a metric graph with $\pi_1(X) = \mathbb F_n$ having edge lengths which sum to $1$, with all vertices of valence at least $3$.
    \item The marking is given by a homotopy equivalence $h:R_n\to X$, where $R_n$ is a graph with one vertex $v$ and $n$ edges. The free group of rank $n$ is identified with $\pi_1(R_n,v)$ where the identification takes generators of the free group to the edges of $R_n$.
     \item Two such marked graphs $(h,X)$ and $(h',X')$ are equivalent if they are isometric via an isometry $g:X\to X'$ such that $g\circ h$ is homotopic to $h'$.
         \end{enumerate}

    In \cite{Ht}, the outer space $CV_n$ is denoted by $\mathbb O_n$. We shall follow the same notation as in \cite{Ht}. An edge $e$ of a graph $X$ is called a separating edge if $X- e$ is disconnected.  There is a natural equivariant deformation retraction of $\mathbb O_n$ onto the subspace $cv_n$ consisting of marked metric graphs $(h,X)$ such that $X$ has no separating edges.  The deformation proceeds by uniformly collapsing all separating edges in all marked graphs in $\mathbb O_n$. The subspace $cv_n$ of $\mathbb O_n$ is also known as \emph{reduced outer space}.

    The group $Out(\mathbb F_n)$ acts on $\mathbb O_n$ on the right by changing the markings: given $\phi\in Out(\mathbb F_n)$, choose a representative $f:R_n \to R_n$ for $\phi$, then $(g,X)\phi = (g\circ f, X)$. This action preserves the reduced outer space $cv_n$ as a subset of $\mathbb O_n$. In particular, we have an $Out(\mathbb F_n)$-action on the reduced outer space $cv_n$.

    In \cite{Ht}, it was shown that the space $\mathbb O_n$ is $Out(\mathbb F_n)$-equivariantly homeomorphic to a subset of sphere complex $\mathbb S(M)$ as described as follows: We interpret points in the sphere complex $\mathbb S(M)$ as (positively) weighted sphere systems in $M$. Let $\mathbb S_\infty$ be the subcomplex of $\mathbb S(M)$ consisting of those elements $\sum_i a_iS_i \in \mathbb S(M)$ such that $M\setminus \cup_i S_i$ has at least one non-simply connected component. To a point $\sum_i a_iS_i \in \mathbb S(M)\setminus \mathbb S_\infty$, we associate the dual graph to $\cup_i S_i$ and assign the edge corresponding to $S_i$ to have length $a_i$. This gives a map $\theta: \mathbb O_n \to \mathbb S(M)\setminus \mathbb S_\infty$. In \cite{Ht}, it is shown that $\theta$ is a homeomorphism.

   Let $N\mathbb S_\infty = NS(M)\cap \mathbb S_\infty$. Then, one can see that the reduced outer space $cv_n$ is $Out(\mathbb F_n)$-equivariantly homeomorphic to $NS(M)\setminus N\mathbb S_\infty $, where the homeomorphism is given by the restriction of $\theta$ to $cv_n\subset \mathbb O_n$.

   Culler and Vogtmann defined what is called the \emph{spine} $K_n$ of the reduced outer space $cv_n$. Considering $cv_n$ as a subset of $NS(M)$, the spine $K_n$ is the maximal full subcomplex of the first barycentric subdivision of $NS(M)$ which is contained in $cv_n$ and is disjoint from $N\mathbb S_\infty$. The interior of every simplex contained in $cv_n$ intersects $K_n$. We have an action of $Out(\mathbb F_n)$ on $NS(M)$ which preserves $cv_n$, i.e., we have an action $Out(\mathbb F_n)$ on $cv_n$. This action of $Out(F_n)$ on $cv_n$ preserves the spine which yields an action of $Out(\mathbb F_n)$ on the spine $K_n$. This gives a homomorphism
   $$\Omega: Out(\mathbb F_n)\to Aut(K_n),$$
    where $Aut(K_n)$ is the group of simplicial automorphisms of $K_n$.

   We have the following theorem of Bridson and Vogtmann (see \cite{BV}),

   \begin{theorem}
   The homomorphism $\Omega$ is an isomorphism for $n\geq 3$.
   \end{theorem}

Now, we have the following lemma:
\begin{lemma}\label{spinelemma}
Every automorphism $\phi\in Aut(NS(M))$ preserves the reduced outer space $cv_n$ and hence the spine $K_n$.
\end{lemma}
\begin{proof}
   We have seen that every automorphism $\phi$ of $NS(M)$ maps simple simplices of $NS(M)$ to simple simplices $NS(M)$, Lemma \ref{simplelemma}. From this it follows that every automorphism $\phi\in Aut(NS(M))$ preserves the reduced outer space $cv_n$ and hence preserves the spine $K_n$ of $cv_n$.
   \end{proof}

    Thus, we  have a homomorphism
   $$\Lambda : Aut(NS(M)) \to Aut(K_n)\cong Out(\mathbb F_n).$$

  \subsection{Proof of Theorem \ref{MainT}}
   Now, Theorem \ref{MainT} follows if we show that the homomorphism $\Lambda$ is injective. In \cite{AS}, it is shown for automorphisms of the sphere complex. The same arguments hold for automorphisms of $NS(M)$. But for the sake of completeness, we write the proof here again.

   \begin{lemma}
   The identity is the only automorphism of $NS(M)$ acting trivially on the spine $K_n$.
   \end{lemma}
   \begin{proof}
   Let $\phi$ be an automorphism of $NS(M)$ which acts trivially on the spine $K_n$. Then, $\phi$ maps every simplex in $cv_n$ to itself. Let $\Sigma=\cup_i \Sigma_i$ be a sphere system in $M$ determining a top dimensional simplex $\sigma$ in $cv_n$. One can see that the codimension $1$ face given by $\Sigma\setminus \Sigma_i$ for each $i$ is also contained in $cv_n$. In order to see this it suffices to prove that it is contained in $NS(M)\setminus N\mathbb S_\infty$.  If this were not the case, then there is a unique component $U$ of $M \setminus(\Sigma \setminus \Sigma_i)$ homeomorphic to $S^2 \times S^1$ with a $3$-ball removed. The boundary of $U$ is a connected component of $\Sigma$ which separates $M$, a contradiction to the assumption that $\sigma\in cv_n$. It follows that the automorphism $\phi$ maps the codimension $1$ face of $\sigma$ determined by $\Sigma$  to itself. In particular, $\phi$ has to fix  the opposite vertex $[\Sigma_i]$ of $\sigma$. As $i$ is arbitrary, this shows that $\phi$ is the identity on $\sigma$. As $\sigma$ is arbitrary, it follows that $\phi$ is the identity on $cv_n$.

 Now, consider any vertex $[S]$ in $NS(M)$. We can extend $S$ to a maximal sphere system $\Sigma$ of non-separating spheres in $M$. The simplex determined by $\Sigma$ is contained in $cv_n$ and is hence fixed by $\phi$. In particular, $\phi$ is identity on $[S]$.  This shows that $\phi$ is identity on $NS(M)$
 \end{proof}

Thus, the group $Aut(NS(M))$ of simplicial automorphisms of $NS(M)$ is isomorphic to the group $Out(F_n)$. In particular, this shows that every automorphism of $NS(M)$ is geometric, i.e., it is given by a homeomorphism of the manifold $M$ by the natural action of the mapping class group.

Now, the $1$-skeleton $\mathbb K$  of the complex $NS(M)$ is a subgraph of the sphere graph. Given any simplicial automorphism of $\phi$ of $\mathbb K$, it has unique extension to a simplicial automorphism $\phi'$ of $NS(M)$ as the complex $NS(M)$ is flag.  The automorphism $\phi'$ is geometric, i.e. there exists a homeomorphism $h$ of the manifold $M$ which induces the automorphism $\phi'$ on $NS(M)$. The homeomorphism $h$ induces an automorphism $\phi''$ on the sphere complex $\mathbb S(M)$ such that restriction of $\phi''$ to $NS(M)\subset \mathbb S(M)$ is $\phi'$. Now, restriction of $\phi''$ to the sphere graph gives an extension of the automorphism $\phi $ of $\mathbb K$ to an automorphism the sphere graph. Thus, we have

\begin{corollary}
Every automorphism $\phi$ of $\mathbb K$ has an extension to an automorphism of the sphere graph.
\end{corollary}

\bibliographystyle{amsplain}

\end{document}